\documentclass{tac}
\usepackage{amsmath,amsfonts,amssymb}
\usepackage{enumerate}
\usepackage[all]{xy}
\usepackage{float}
\usepackage{graphicx}

\newtheorem{note}{Note}
\newtheorem{construction}[theorem]{Construction}
\newcommand{\dom}{\mathrm{dom}}
\newcommand{\cod}{\mathrm{cod}}
\newcommand{\meet}{\wedge}
\newcommand{\join}{\vee}
\newcommand{\catG}{\mathbf{G}}
\newcommand{\X}{\mathbf{X}}
\newcommand{\Y}{\mathbf{Y}}

\title{The Ehresmann-Schein-Nambooripad Theorem for inverse categories}
\author{Darien DeWolf and Dorette Pronk}
\address{Department of Mathematics, Statistics and Computer Science\\
St. Francis Xavier University\\2323 Notre Dame Lane\\
Antigonish, NS B2G 1N5\\Canada\\[5pt]
Department of Mathematics and Statistics\\
Dalhousie University\\
Halifax, NS B3H 4R2\\Canada\\}
\eaddress{ddewolf@stfx.ca\CR dorette.pronk@dal.ca}
\copyrightyear{2017}

\thanks{Both authors thank NSERC for its funding of this research.}

\begin{document}
\maketitle

\begin{abstract}
The Ehresmann-Schein-Nambooripad (ESN) Theorem asserts an equivalence
between the category of inverse semigroups and the category of inductive
groupoids. In this paper, we consider the category of inverse categories
and functors -- a natural generalization of inverse semigroups --  and
extend the ESN theorem to an equivalence between this category and the
category of top-heavy locally inductive groupoids and locally inductive
functors. From the proof of this extension, we also generalize the ESN
Theorem to an equivalence between the category of inverse semicategories
and the category of  locally inductive groupoids and to an equivalence
between the category of inverse categories with oplax functors and the
category of top-heavy locally inductive groupoids and ordered functors.
\end{abstract}

\section{Introduction}

The Ehresmann-Schein-Nambooripad (ESN) Theorem asserts the existence of
an equivalence between the category of inverse semigroups (with
semigroup homomorphisms) and the category of inductive groupoids (with
inductive functors). A groupoid is called ordered in this context if
there is a compatible (functorial) order on both objects and arrows with
a notion of restriction on the arrows such that an arrow
$f\colon A\to B$ has a unique restriction $f'\colon A'\to B'$ with
$f'\leq f$ for any object $A'\leq A$. For the precise definition see
Definition~\ref{def:orderedgroupoid}, but a category theorist may like
to think of these as groupoids internal to the category of posets with
some additional properties. Ordered functors between these are functors
that preserve the order. Furthermore, an ordered groupoid is called
inductive when the objects form a meet-semilattice and an ordered
functor is inductive when it preserves the meets. The correspondence of
the ESN Theorem is directly extendable to inverse semigroups and
prehomomorphisms when one takes ordered functors, rather than inductive
functors, between the inductive groupoids.

This theorem has been extended to various larger classes of semigroups
such as regular semigroups \cite{nambooripad1, nambooripad2,
nambooripad3}, two-sided restriction semigroups (also called Ehresmann
semigroups) \cite{Lawson:1991jt} and more general restriction groups
\cite{Hollings:2010iq} with either semigroup homomorphisms or $\meet$-
or $\join$-prehomomorphisms. The main ideas in this context have focused
on either changing the requirement for a meet-semilattice structure to a
different order structure on the objects of the groupoid, or on
generalizing to inductive categories rather than groupoids.

Our approach here is to generalize this equivalence in a different
direction. Semigroups can be viewed  as single-object semicategories and
we want to obtain a `multi-object' version of the correspondence.
As groupoids can be thought of as the multi-object version of groups,
we think of inverse categories as a multi-object version of inverse
semigroups. In this paper, we prove a new generalization of the ESN
theorem which extends the result to inverse categories. Since we are
generalizing the concept of {\em inverse} semigroup, we will remain
within the category of groupoids. They will still be ordered, but the
order structure will only be locally inductive in a suitable sense: the
objects need to form a disjoint union of meet-semilattices. Since
inverse categories have identities, we further require that  the
meet-semilattices have a top-element. If we instead generalize to
inverse semicategories, this requirement is not needed. Locally
inductive functors, ordered functors that preserve all meets that exist,
will correspond to functors of inverse semicategories
(Corollary~\ref{cor:irscat_og_equiv}). We will also show that the
category of inverse categories and oplax functors is equivalent to the
category of top-heavy locally inductive groupoids and locally inductive
functors, generalizing the classical result that the category of inverse
semigroups and prehomomorphisms is equivalent to the category of
inductive groupoids and ordered functors (Theorem~\ref{thm:oplax_is}).

In private communication, Lawson provided an unpublished preprint in
which he provides a similar construction to provide a proof of our main
result. His constructions rely on the existence of (maximal) identities
in an inverse category, while ours relies on the partitioning of the
objects into meet-semilattices. The advantage of our approach is that
the ESN Theorem is a direct corollary of the equivalence between
top-heavy locally inductive groupoids by simply removing identities and
applying our construction to a single-object inverse semicategory.

The groupoid we construct for an inverse category was also independently
considered in the work of Linckelmann \cite{Linckelmann:2013gd} on
category algebras. Linckelmann observes that this groupoid has the same
category algebra as the original inverse category, giving the category
algebra of an inverse category the structure of a groupoid algebra:
a groupoid algebra over a commutative ring is a direct product of matrix
rings. In this paper, we introduce this groupoid with an ordered
structure  and observe the important characterizing properties of the
order structure to obtain an equivalence of categories between the
category of inverse categories and the category of these top-heavy
locally inductive groupoids.

From the semigroup perspective, this raises the question of whether
there are appropriate multi-object versions of the other classes of
semigroups mentioned above which then may be shown to be equivalent to
appropriate versions of locally inductive categories.

\subsection{Inductive groupoids and inverse semigroups}

Inductive groupoids are a class of groupoids whose arrows are equipped
with a partial order satisfying certain conditions and whose objects
form a meet-semilattice. Charles Ehresmann used ordered groupoids to
model pseudogroups while inverse semigroups, introduced by Gordon
Preston \cite{preston1}, were concurrently used as an alternate model
for pseudogroups. Ehresmann was certainly aware of the connection
between ordered (inductive) groupoids and inverse semigroups, as it was
Ehresmann who first introduced the tensor product required to make the
correspondence work.
Boris Schein \cite{schein1} made this connection explicit, requiring
that the set of objects form a meet-semilattice, thus guaranteeing the
existence of this tensor product for all arrows of the groupoid.
K.S.S.~Nambooripad \cite{nambooripad1, nambooripad2, nambooripad3,
nambooripad1979} independently developed the theory of so-called regular
systems and their correspondence to so-called regular groupoids.
This theory is, in fact, more general and specializes to the
correspondence of inverse semigroups to inductive groupoids. A more
detailed history of inverse semigroups, inductive groupoids and their
applications can be found in Hollings' \cite{Hollings:2009uh}. In this
section, we present the modern exposition of this correspondence, which
can be found in Mark Lawson's book \cite{lawson98}.

\begin{definition} \label{def:orderedgroupoid}
A groupoid $\catG$ is said to be an \emph{ordered groupoid} whenever
there is a partial order $\leq$ on its arrows satisfying the following
four conditions:
\begin{enumerate}[(i)]
\item\label{def:orderedgroupoid_inversepreserve}
For each arrow $f,g\in G,$ $f \leq g$ implies $f^{-1} \leq g^{-1}.$
\item\label{def:orderedgroupoid_compositionpreserve}
For all arrows $a,A,b,B\in G$ such that $a\leq A,$ $b\leq B$ and the
composites $ab$ and $AB$ exist, $ab \leq AB.$
\item\label{def:orderedgroupoid_restriction}
For each arrow $f:A'\rightarrow B$ in $G$ and each object $A \leq A'$ in
$G,$ there exists a unique \emph{restriction of $f$ to $A$}, denoted
$[f |_* A],$ such that $\dom[f |_* A] = A$ and $[f |_* A]\leq f.$
\item\label{def:orderedgroupoid_corestriction}
For each arrow $f:A\rightarrow B'$ in $G$ and objects $B \leq B'$ in
$G,$ there exists a unique \emph{corestriction of $f$ to $B$}, denoted
$[B\,{}_*| f],$ such that $\cod[B\,{}_*| f] = B$ and
$[B\,{}_*| f]\leq f.$
\end{enumerate}
An ordered groupoid is said to be an \emph{inductive groupoid} whenever
its objects form a meet-semilattice.
\end{definition}

Though it is sometimes convenient to explicitly give both the
restrictions and corestrictions in an ordered groupoid, the following
proposition makes it necessary only to include one of them in any
proofs.

\begin{proposition}[\cite{lawson98}]
In Definition~\ref{def:orderedgroupoid}, conditions
(\ref{def:orderedgroupoid_restriction}) and
(\ref{def:orderedgroupoid_corestriction}) are equivalent.
\end{proposition}

\begin{definition} \label{def:orderedgroupoid_tensor}
Let $\mathbf{G}$ be an ordered groupoid with arrows
$\alpha,\beta\in \mathbf{G}.$ If $\dom(\alpha)\meet \cod(\beta)$ exists,
the \emph{tensor product} $\alpha\otimes \beta$ of $\alpha$ and $\beta$
is defined as
\[ \alpha\otimes \beta = [\alpha \, |_* \, \dom(\alpha) \meet
   \cod(\beta)][\dom(\alpha) \meet \cod(\beta) \, {}_*|\, \beta]. \]
\end{definition}

\begin{proposition}[\cite{lawson98}] \label{prop:associative_tensor}
Let $\catG$ be an inductive groupoid.
This tensor product is associative and admits pseudoinverses given by
the inverses in the inductive groupoid, making $(\mathbf{G}_1, \otimes)$
an inverse semigroup.
\end{proposition}

\begin{proof}[Proof sketch.]
For any pair of arrows in $\catG,$ one can show that the set
\[ \langle \alpha, \beta\rangle = \{(\alpha',\beta')|\,\cod(\alpha') =
   \dom(\beta'), \alpha'\leq \alpha, \beta'\leq \beta\} \]
contains a unique maximal element $(\alpha', \beta')$ with
$\alpha\otimes\beta = \beta' \circ \alpha'.$
Since defined by composition, this tensor product is therefore
associative.
\end{proof}

\begin{proposition} \label{prop:identity_restrictions}
For all objects $A \leq B$ of an ordered groupoid,
$[1_B \,|_*\, A] = 1_A = [A \,{}_*|\, 1_B].$
\end{proposition}
\begin{proof}
The partial order on arrows induces the partial order on the objects of
an ordered groupoid. Since an object of a category is identified by the
identity arrow on that object, we have that $1_A \leq 1_B.$ Since the
(co)domain of $1_A$ is $A,$ we have
$[1_B \,|_*\, A] = 1_A = [A \,{}_*|\, 1_B]$ by the uniqueness of
(co)restrictions
\end{proof}

\begin{definition} \label{def:ordered_functor}
A morphism $F: \mathbf{G} \rightarrow \mathbf{H}$ of ordered groupoids
(an \emph{ordered functor}) is a functor such that, for all arrows
$f \leq g$ in $\mathbf{G},$ $F(f) \leq F(g)$ in $\mathbf{H}.$ An ordered
functor between inductive groupoids is said to be \emph{inductive}
whenever it preserves the meet structure on objects.
\end{definition}

\begin{notation}
We denote the category of ordered groupoids and ordered functors by
$\mathbf{oGrpd}$ and the category of inductive groupoids and inductive
functors by $\mathbf{iGrpd}.$
\end{notation}

We will now briefly review Lawson's description of functorial
constructions that form the equivalence of categories between the
category of inverse semigroups and the category of inductive groupoids.
We remind the reader that full details can be found in \cite{lawson98}.

\begin{construction}[Inverse Semigroups to Inductive Groupoids]
Given an inverse semigroup $(S, \bullet),$ define an inductive groupoid
$\mathcal{G}(S)$ with the following data:
\begin{itemize}
\item Objects: $\mathcal{G}(S)_0 = E(S),$ the idempotents in $S.$ Since
      $S$ is an inverse semigroup, $E(S)$ is a meet-semilattice with
      meets given by the product in $S.$
\item Arrows: For each element $s\in S,$ there is an arrow
      $s : s^\bullet s \rightarrow ss^\bullet$ (we remind the reader
      that $s^\bullet$ denotes the partial inverse of $s).$ Composition
      is given by multiplication in $S$ and identities are the elements
      of $E(S).$
\item Inverses: For each arrow $s : s^\bullet s \rightarrow ss^\bullet$
      in $\mathcal{G}(S),$ define $s^{-1} = s^\bullet,$ its
      pseudoinverse in $S.$
\item The partial order on arrows is given by the natural partial order
      ($s\leq t$ if and only if $s = te$ for some idempotent $e)$ on the
      elements of $S.$ It can be checked that this partial order
      satisfies conditions (\ref{def:orderedgroupoid_inversepreserve})
      and (\ref{def:orderedgroupoid_compositionpreserve}) of an ordered
      groupoid.
\item The (co)restrictions are also given by multiplication in $S.$
      This can be checked to satisfy condition
      (\ref{def:orderedgroupoid_restriction}) of an ordered groupoid.
\end{itemize}
\end{construction}

\begin{construction}[Inductive Groupoids to Inverse Semigroups]
Given an inductive groupoid $(\catG, \leq),$ define an inverse semigroup
$\mathcal{S}(\catG)$ whose elements are the arrows of $\catG$ and whose
multiplication is given by the tensor product. This is an inverse
semigroup operation with inverses those from $\catG$
(Proposition~\ref{prop:associative_tensor}).
\end{construction}

\begin{theorem}[ESN, \cite{lawson98}]
\label{thm:ESN}
The constructions $\mathcal{G}$ and $\mathcal{S}$ are functorial and
form an equivalence of categories
\[
\xymatrix{
\mathbf{iGrpd} \ar[r]<+0.3pc>^-{\mathcal{S}} &
\mathbf{iSgp} \ar[l]<+0.3pc>^-{\mathcal{G}}
}
\]
\end{theorem}

\subsection{Inverse categories as restriction categories}

\begin{definition}[\cite{Cockett:2002te}]
\label{def:restriction_category}
A restriction structure on a category $\mathbf{X}$ is an assignment of
an arrow $\overline{f}:A\rightarrow A$ to each arrow $f:A\rightarrow B$
in $\mathbf{X}$ satisfying the following four conditions:
\begin{itemize}
\item[(R.1)] \label{def:rc_R1} For all maps $f,$ $f\,\overline{f} = f.$
\item[(R.2)] \label{def:rc_R2} For all maps $f:A\rightarrow B$ and
  $g: A\rightarrow B',$
  $\overline{f}\,\overline{g} = \overline{g}\,\overline{f}.$
\item[(R.3)] \label{def:rc_R3} For all maps $f:A\rightarrow B$ and
  $g:A\rightarrow B',$
  $\overline{g\,\overline{f}} = \overline{g}\,\overline{f}.$
\item[(R.4)] \label{def:rc_R4} For all maps $f:B\rightarrow A$ and
  $g:A\rightarrow B',$
  $\overline{g}\, f = f \,\overline{gf}.$
\end{itemize}
A category equipped with a restriction structure is called a
\emph{restriction category}.
\end{definition}

\begin{definition}[\cite{Cockett:2002te}]
A \emph{restriction functor} $F: \X \rightarrow \Y$ between restriction
categories is a functor
which preserves the restriction idempotents;
$F\left(\overline{f}\right) = \overline{F(f)}$ for all $f\in \X_1.$
\end{definition}

\begin{example} \label{ex:rcats}
Examples of restriction categories:
\begin{enumerate}[(a)]
\item
$\mathbf{Par}$ as defined above is the prototypical example of a
restriction category. The axioms (R.1) -- (R.4) required for
$\mathbf{Par}$ to be a restriction category are easily verified. We can
interpret expressions such as $f\,\overline{g}$ as ``$f$ restricted to
where $g$ is defined''.
\item\label{ex:stable_monics}
Let $\mathbf{C}$ be an ordinary category equipped with a stable system
$\mathcal{M}$ of monics (all details of this example can be found in
\cite{Cockett:2002te}). Define a category
$\mathrm{Par}(\mathbf{C},\mathcal{M})$ with the following data:
\begin{itemize}
\item Objects: Same objects as $\mathbf{C}.$
\item Arrows: Isomorphism classes of spans
\[
\xymatrix{
X & \,\,D \ar@{>->}[l]_i \ar[r]^f & Y,
}\]
where $i \in \mathcal{M}.$ We will sometimes denote such an arrow
(actually, its isomorphism class) as $(i,f).$
\item Composition: Composition is given by pullback.
\item Restrictions: Given any arrow $(i,f),$ the assignment
$\overline{(i,f)} = (i,i)$ defines a restriction structure on
$\mathrm{Par}(\mathbf{C},\mathcal{M}).$
\end{itemize}
\end{enumerate}
\end{example}

The next lemma lists some useful identities that will be used, without
reference, to make calculations in restriction categories.
\begin{lemma}[\cite{Cockett:2002te}]
\label{lem:rc_props}
If $\mathbf{X}$ is a restriction category, then:
\begin{enumerate}[(i)]
\item \label{lem:rc_props_i}
  $\overline{f}$ is idempotent;
\item \label{lem:rc_props_ii}
  $\overline{f}\,\overline{gf} = \overline{gf};$
\item \label{lem:rc_props_iii}
  $\overline{\overline{g}f} = \overline{gf};$
\item \label{lem:rc_props_iv}
  $\overline{\overline{f}} = \overline{f};$
\item \label{lem:rc_props_v}
  $\overline{\overline{g}\overline{f}} = \overline{g}\overline{f};$
\item \label{lem:rc_props_vi}
  if $f$ is monic, then $\overline{f} = 1;$
\item \label{lem:rc_props_vii}
  $f\overline{g} = f$ implies $\overline{f} = \overline{f}\overline{g}.$
\end{enumerate}
\end{lemma}

\begin{note}
A restriction category $\mathbf{X}$ has a natural, locally partially
ordered 2-category structure: for any two parallel arrows
$f,g:C\rightarrow D$ in $\mathbf{X},$ we define a partial order by
$f\leq g$ if and only if $f = g\overline{f}.$ Notice that if $f\leq g,$
then
\[ \overline{g}\overline{\overline{f}}
  = \overline{g}\overline{f}
  = \overline{g\overline{f}}
  = \overline{f} \]
and thus $\overline{f}\leq \overline{g}.$
\end{note}

\begin{proposition} \label{prop:2cat_properties}
\label{prop:2cat_properties_preserved_by_composition} Suppose that
$a, A, b$ and $B$ are arrows in a restriction category $\mathbf{X}$ with
$a\leq A$ and $b\leq B.$ If the composites $ab$ and $AB$ exist, then
$ab\leq AB.$
\end{proposition}

\begin{proof}
Suppose that $a,A,b$ and $B$ are arrows in $\mathbf{X}$ with $a\leq A,$
$b\leq B$ and such that the composites $ab$ and $AB$ exist. Then
\[ AB\overline{ab}
  = AB\overline{b}\,\overline{ab}
  = Ab\overline{ab} = A\overline{a}b
  = ab \]
and thus $ab \leq AB.$
\end{proof}

\begin{definition} \label{def:total}
A map $f$ in a restriction category $\mathbf{X}$ is called \emph{total}
whenever $\overline{f} = 1.$
\end{definition}

\begin{lemma}[\cite{Cockett:2002te}]
\label{lem:total_props}
If $\mathbf{X}$ is a restriction category, then:
\begin{enumerate}[(i)]
\item every monomorphism is total;
\item if $f$ and $g$ are total, then $gf$ is total;
\item if $gf$ is total, then $f$ is total;
\item the total maps form a subcategory, denoted
      $\mathrm{Tot}(\mathbf{X}).$
\end{enumerate}
\end{lemma}

\begin{definition} \label{def:rc_morphism}
A morphism $F: \mathbf{X} \rightarrow \mathbf{Y}$ of restriction
categories (a \emph{restriction functor}) is a functor such that
$F\left(\overline{f}\right) = \overline{F(f)}$ for each $f\in X_1.$
\end{definition}

\subsubsection*{Inverse categories}

As groupoids are for groups, we will use a structure describing
multi-object inverse semigroups. Inverse semigroups with units are
exactly single-object inverse categories, so it seems that inverse
(semi)categories could be appropriate for such a role.

\begin{definition}[\cite{kastl1979}]
\label{def:inverse_category_nonrestriction}
A category $\mathbf{X}$ is said to be an \emph{inverse category}
whenever, for each arrow $f:A \rightarrow B$ in $\mathbf{X},$ there
exists a unique $f^\circ : B \rightarrow A$ in $\mathbf{X}$ such that
$f\circ f^\circ \circ f = f$ and
$f^\circ \circ f \circ f^\circ = f^\circ.$
\end{definition}

\begin{definition} \label{def:restricted_inverse_isomorphism}
A map $f$ in a restriction category $\mathbf{X}$ is called a
\emph{restricted isomorphism} whenever there exists a map $ g$ -- called
a \emph{restricted inverse} of $f$ -- such that $gf = \overline{f}$ and
$fg = \overline{g}.$
\end{definition}

Following from the commutation of idempotents (Restriction Category
Axiom (\ref{def:rc_R2}), we have the following property of restricted
isomorphisms:
\begin{theorem}[Lemma 2.18(vii), \cite{Cockett:2002te}]
If $f$ is a restricted isomorphism, then its restricted inverse is
necessarily unique.
\end{theorem}

\begin{note}
If a category $\X$ has the \emph{property} of being an inverse category,
one can define a restriction \emph{structure} on $\X$ by defining
$\overline{f} = f^\circ f.$ Indeed, with this restriction structure,
every arrow in $\X$ is a restricted isomorphism and the restricted
inverse of an arrow $f$ is exactly $f^\circ.$ This justifies the
following notation and definition.
\end{note}

\begin{notation}
Given a map $f$ in a restriction category $\mathbf{X},$ we denote its
restricted inverse (if it exists) by $f^\circ.$
\end{notation}

\begin{definition} \label{def:inverse_category}
A restriction category $\mathbf{X}$ is called an
\emph{inverse category}, whenever every map $f$ is a restricted
isomorphism.
\end{definition}

\begin{example} Some inverse categories:
\begin{enumerate}[(a)]
\item The category of sets and partial bijections.
\item Any inverse semigroup with unit is a single-object inverse
      category.
\item Any groupoid is an inverse category with all arrows total.
\end{enumerate}
\end{example}

\begin{lemma}[\cite{Cockett:2002te}]
\label{lem:restriction_functor_props}
If $F:\mathbf{X}\rightarrow \mathbf{Y}$ is a restriction functor, then
$F$ preserves
\begin{enumerate}[(i)]
\item total maps,
\item restriction idempotents,
\item restricted sections and
\item restricted isomorphisms.\label{lem:restriction_functor_props_isos}
\end{enumerate}
\end{lemma}

\begin{note}
Any functor between inverse categories is a restriction functor
preserving restricted isomorphisms. This follows from the restriction
structure  and restricted isomorphisms being defined as specific
composites. We will therefore omit the words ``inverse'' and
``restriction'' when speaking of functors between inverse categories.
\end{note}

As expected, restriction idempotents are their own restricted inverse.

\begin{proposition} \label{prop:inverse_of_restriction}
In an inverse category,
$\left(\,\overline{f}\,\right)^\circ = \overline{f}$ for all arrows $f.$
\end{proposition}
\begin{proof}
Since all arrows in an inverse category are restricted isomorphisms,
\[ \left(\,\overline{f}\,\right)^\circ
  = (f^\circ f)^\circ
  = f^{\circ}(f^{\circ})^{\circ}
  = f^\circ f
  = \overline{f}. \]
\end{proof}

It is clear that inverse categories, interpreted as restriction
categories in Definition~\ref{def:inverse_category}, are exactly the
same as inverse categories interpreted as multi-object inverse
semigroups in Definition~\ref{def:inverse_category_nonrestriction}; that
is restrictions come for free in an inverse category and are given by
$\overline{f} = f^\circ f.$ In this paper, we choose to think in terms
of restriction categories for two reasons: firstly, the choice of
notation in restriction categories facilitates calculations. Secondly,
we prefer to think of (finite) inverse semigroups as collections of
partial automorphisms on a (finite) set whose idempotents are partial
identities -- inverse categories in terms of restriction categories
explicitly make use of this intuition.

\begin{notation}
We denote the category of inverse categories and functors by
$\mathbf{iCat}.$
\end{notation}

\section{Main result}

In this section, we introduce the notion of top-heavy locally inductive
groupoids: ordered groupoids whose objects may be partitioned into
meet-semilattices, each of which contain a top element. We will then
give functorial constructions of top-heavy locally inductive groupoids
from inverse categories, and vice versa. These constructions will then
be seen to give an equivalence of categories between $\mathbf{iCat}$ and
$\mathbf{tliGrpd}$ (the category of top-heavy locally inductive
groupoids). The identities of an inverse category are seen to correspond
to the tops of the meet-semilattices in a top-heavy locally inductive
groupoid and the equivalence can thus be immediately generalized to an
equivalence between the category of inverse semicategories and
semifunctors and the category of locally inductive groupoids and locally
inductive functors. Finally, we end this section with a short discussion
of a categorical analogue of the classical result in semigroup
theory that the category of inverse semigroups and prehomomorphisms is
equivalent to the category of inductive groupoids and ordered functors.
Explicitly, we show that the category of inverse categories and oplax
functors is equivalent to the category of top-heavy locally inductive
groupoids and ordered functors.
\begin{definition}
Let $A$ be an object of a restriction category $\mathbf{X}.$ Let $E_A$
denote the set of restrictions of all endomorphisms on $A.$ That is,
\[ E_A = \left\{\overline{f}:A\rightarrow A |
                f: A\rightarrow A \in \mathbf{X} \right\}.\]
\end{definition}
Notice that, for any $f: A \rightarrow B$ in $\mathbf{X},$ we have
$\overline{f}:A \rightarrow A \in E_A,$ since
$\overline{\overline{f}} = \overline{f}.$ The reason for specifying that
the restrictions in $E_A$ come from endomorphisms in $\mathbf{X},$ then
serves no use further than simply reminding us that the equivalence we
are trying to establish here is based on the observation that an inverse
category is, at each object, an inverse semigroup (with identity).

\begin{proposition} \label{prop:EA_meetsemilattice}
For each object $A$ of a restriction category $\mathbf{X},$ $E_A$ is a
meet-semilattice with meets given by
$\overline{a} \meet \overline{b} = \overline{a}\overline{b}.$
In addition, $E_A$ has top element $1_A.$
\end{proposition}
\begin{proof}
First of all, $E_A$ is a poset with the natural partial order inherited
from $\mathbf{X}.$ We now show that $E_A$ has finite meets given by
$\overline{a}\meet \overline{b} = \overline{a}\overline{b}:$
\begin{itemize}
\item First, it is a lower bound:
\[\overline{a}\,\overline{\overline{a}\meet\overline{b}}
  = \overline{a}\,\overline{\overline{a}\overline{b}}
  =  \overline{a}\,\overline{a}\overline{b}
  = \overline{a}\,\overline{b}
  = \overline{a}\meet\overline{b} \]
and thus $\overline{a}\meet\overline{b} \leq \overline{a}.$ Similarly,
$\overline{a}\meet\overline{b} \leq \overline{b}.$
\item This lower bound is unique up to isomorphism (equality): suppose
that $\overline{d}$ is such that $\overline{d} \leq \overline{a},$
$\overline{d}\leq \overline{b}$ and
$\overline{a} \meet \overline{b} \leq \overline{d}.$
Then
\[ \overline{d} = \overline{a}\,\overline{d}
  = \overline{a}\,\overline{b}\,\overline{d}
  = \overline{d}\,\overline{\overline{a}\,\overline{b}}
  = \overline{d}\,\overline{\overline{a}\meet \overline{b}}
  = \overline{a}\meet \overline{b}.  \]
\end{itemize}
Finally, since $\overline{1_A} = 1_A,$ $1_A \in E_A.$ Also, given any
$\overline{f}: A\rightarrow A,$ $1_A \overline{f} = \overline{f}$ and
thus $\overline{f} \leq 1_A$ and $1_A$ is the top element of $E_A.$
\end{proof}

\begin{proposition} \label{prop:E_A_disjoint}
For each pair of objects $A$ and $B$ of a restriction category
$\mathbf{X},$  if $A \neq B,$ then
$$E_A \cap E_B = \varnothing.$$
\end{proposition}
\begin{proof}
If $\overline{f} \in E_A \cap E_B,$ then
$A = \dom \left(\overline{f}\right) = B.$
\end{proof}

We may now give the (functorial) constructions giving an equivalence
between the category of top-heavy locally inductive groupoids and
inverse categories.
\begin{construction}
\label{conts:orderedgroupoid}
Given an inverse category
$\left(\mathbf{X},\circ,\overline{(-)}\right),$
define a groupoid\\
$(\mathcal{G}(\mathbf{X}),\bullet,\leq)$ with the following data:
\begin{itemize}
\item Objects:
  $\displaystyle\mathcal{G}(\mathbf{X})_0
  = \coprod_{A\in \mathbf{X_0}} E_A.$
\item Arrows:
  Every arrow in $\mathcal{G}(\mathbf{X})$ is of the form
  $f : \overline{f_A} \rightarrow \overline{f^\circ_B}$ for each arrow
  $f: A \rightarrow B$ in $\mathbf{X}.$
\begin{itemize}
\item Composition:
  for arrows $f:\overline{f}\rightarrow \overline{f^\circ}$ and
  $g:\overline{g}\rightarrow \overline{g^\circ}$ with
  $\overline{f^\circ}=\overline{g},$ we define their composite
  $g\bullet f:\overline{f}\rightarrow \overline{g^\circ}$ in
  $\mathcal{G}(\mathbf{X})$ to be their composite in $\mathbf{X}.$
  This composite is indeed an arrow, for
  \[ \overline{gf} = \overline{\overline{g}f}
    = \overline{\overline{f^\circ}f}= \overline{f}\]
  and
  \[ \overline{(gf)^\circ} = \overline{f^\circ g^\circ}
    = \overline{\overline{f^\circ} g^\circ}
    = \overline{\overline{g}g^\circ}
    = \overline{g^\circ}. \]
\item Identities:
  For any object $\overline{f}:A\rightarrow A$ in
  $\mathcal{G}(\mathbf{X}),$ define
  $1_{\overline{f}} = \overline{f}$ (which is well-defined since
  $\overline{\overline{f}}=\overline{f}$). The identity then satisfies
  the appropriate axiom: for each
  $g : \overline{g} \rightarrow \overline{g^\circ}$ with
  $\overline{g} = \overline{f}$ and
  $\overline{g^\circ} = \overline{f^\circ},$ we have
  $\overline{f^\circ} g = \overline{g^\circ}g = g$ and
  $g\overline{f} = g\overline{g} = g.$
\item Inverses:
  Given an arrow $f : \overline{f}  \rightarrow \overline{f^\circ},$
  define $f^{-1}: \overline{f^\circ} \rightarrow \overline{f}$ to be
  $f^\circ,$ the unique restricted inverse of $f$ from $\mathbf{X}$'s
  inverse structure. The composites are
  $f f^\circ = \overline{f^\circ} =1_{\overline{f^\circ}}$ and
  $f^\circ f = \overline{f} =1_{\overline{f}}$ as required.
\end{itemize}
\end{itemize}
\end{construction}

\begin{definition} \label{def:top_heavy_locally_inductive_groupoid}
An ordered groupoid is said to be a \emph{locally inductive groupoid}
whenever there is a partition $\{M_i\}_{i\in I}$ of $\mathbf{G}_0$ into
meet-semilattices $M_i$ with the property that any two comparable
objects be in the same meet-semilattice $M_i.$ A locally inductive
groupoid is said to be \emph{top-heavy} whenever each meet-semilattice
$M_i$ admits a top-element $\top_i.$
\end{definition}

\begin{note}
The requirement that any two comparable objects of a locally inductive
groupoid be in the same meet-semilattice corresponds to our intuition
that if the meet $A\meet B$ of two objects $A$ and $B$ exists in $M_i,$
then $A$ and $B,$ both sitting above this meet, should also be elements
of $M_i.$
\end{note}

\begin{definition}
An ordered functor between locally inductive groupoids is said to be
\emph{locally inductive} whenever it preserves all meets that exist.
In particular, a locally inductive functor will preserve empty meets and
thus top elements and there is no requirement to define so-called
``top-heavy locally inductive functors''.
\end{definition}

\begin{notation}
We denote the category of locally inductive groupoids and locally
inductive functors by $\mathbf{liGrpd}$ and the category of top-heavy
locally inductive groupoids and locally inductive functors by
$\mathbf{tliGrpd}.$
\end{notation}

\begin{proposition} \label{prop:G(X)_is_OG}
For each inverse category $\mathbf{X},$ $\mathcal{G}(\mathbf{X})$ is a
top-heavy locally inductive groupoid.
\end{proposition}
\begin{proof}
Recall that the partial order on the objects $\overline{f}$ in
$\mathcal{G}(\mathbf{X})$ is that which is induced by the partial order
on the arrows of $\mathbf{X}.$ That is, $\overline{f} \leq \overline{g}$
if and only if $\overline{f} = \overline{g}\overline{\overline{f}} =
\overline{g}\overline{f}.$ We now prove that this partial order gives
$\mathcal{G}(\mathbf{X})$ the structure of an ordered groupoid:
\begin{enumerate}[(i)]
\item Suppose that $f$ and $g$ are arrows in $\mathcal{G}(\mathbf{X})$
with $f \leq g.$ That is, we suppose that $g\overline{f} = f$
(since these are also arrows in $\mathbf{X}).$ Then
\begin{align*}
f^\circ &= (g\overline{f})^\circ
= \overline{f}^\circ g^\circ
= \overline{f} g^\circ
= \overline{f} g^\circ g g^\circ
= g^\circ g \overline{f} g^\circ \\
&= g^\circ g \overline{f}\,\overline{f} g^\circ
= g^\circ g \overline{f}\,\overline{f}^\circ g^\circ
= g^\circ g \overline{f}(g\overline{f})^\circ
= g^\circ  {f} f^\circ \\
&= g^\circ \overline{f^\circ}
\end{align*}
and thus $f^{-1} = f^\circ \leq g^\circ = g^{-1}.$
\item This follows directly from Proposition~\ref{prop:2cat_properties}.
\item Given an arrow
  $\alpha : \overline{\alpha}\rightarrow \overline{\alpha^\circ}$ with
  an object $\overline{e}\leq \overline{\alpha},$ we define the
  restriction $[\alpha |_* \overline{e}]$ of $\alpha$ to $\overline{e}$
  to be $\alpha \overline{e}.$ This is indeed an arrow whose domain is
  $\overline{e}:$ $\overline{\alpha\overline{e}} =
   \overline{\alpha}\,\overline{e}=  \overline{e}.$

Also, $\alpha \overline{\alpha\overline{e}} =
       \alpha \overline{\alpha}\,\overline{e} =
       \alpha\overline{e},$ so that $\alpha\overline{e} \leq \alpha.$

If $\beta \leq \alpha$ is any other arrow with
$\dom(\beta) = \overline{e},$ we have $\alpha\overline{\beta} = \beta$
and $\overline{\beta} = \overline{e},$ so that
$\beta = \alpha\overline{e}$ and thus $[\alpha |_* \overline{e}]$ as
defined is unique.
\item Given an arrow
$\alpha : \overline{\alpha}\rightarrow \overline{\alpha^\circ}$ with an
object $\overline{e}\leq \overline{\alpha^\circ},$ we define the
corestriction $[\overline{e}\,{}_*| \alpha]$ of $\alpha$ to
$\overline{e}$ to be $\overline{e}\alpha .$ This is indeed an arrow
whose codomain is $\overline{e}:$
$\overline{(\overline{e}\alpha)^\circ}
  = \overline{\alpha^\circ \overline{e}}
  = \overline{a^\circ}\,\overline{e}
  = \overline{e}.$

Also,
$ \alpha \overline{\overline{e}\alpha} =
  \alpha \overline{e\alpha} = \alpha(e\alpha)^\circ e\alpha =
  \alpha \alpha^\circ e^\circ e \alpha =
  e^\circ e \alpha \alpha^\circ \alpha =
  \overline{e}\alpha,$ so that $\overline{e}\alpha \leq \alpha.$

If $\beta \leq \alpha$ is any other arrow with
$\cod(\beta) = \overline{e},$ we have $\beta^\circ \leq \alpha^\circ$
(property (i) of ordered groupoids) and thus
$\alpha^\circ\overline{\beta^\circ} = \beta^\circ$ and
$\overline{\beta^\circ} = \overline{e},$ so that
$\beta^\circ = \alpha^\circ \overline{e} = (\overline{e}\alpha)^\circ$
and thus $[\overline{e}\,{}_*| \alpha]$ as defined is unique.
\end{enumerate}
Given the choice of objects for $\mathcal{G}(\mathbf{X}),$ it follows
immediately from Propositions~\ref{prop:EA_meetsemilattice}
and \ref{prop:E_A_disjoint} that $\mathcal{G}(\mathbf{X})$ is a
top-heavy locally inductive groupoid.
\end{proof}
The composition in $\mathcal{G}(\mathbf{X})$ of $f$ and $g$ exists
exactly when $\overline{f}  = \overline{g^\circ}$ and is defined by the
composition in $\mathbf{X}.$ The tensor product in
$\mathcal{G}(\mathbf{X})$ is a natural extension of this composition in
the sense that it exists whenever the meet
$\overline{f} \meet \overline{g^\circ}$ exists. This lemma shows that
this extension is also defined by the composition in $\mathbf{X}.$

\begin{lemma} \label{lem:tensor_is_composition}
If $\mathbf{X}$ is an inverse category, then in
$\mathcal{G}(\mathbf{X})$ the tensor products (when defined) are given
by composition in $\mathbf{X}:$
\[ f \otimes g = f g.\]
\end{lemma}
\begin{proof} Recall that, for any arrow $f$ in $\mathbf{X},$
$\dom(f) = \overline{f}$ and $\cod(f) = \overline{f^\circ}.$
Then
\begin{align*}
f \otimes g
&=
  \left[f\,|_*\, \dom(f)\meet \cod(g)\right]
  \left[\dom(f)\meet \cod(g)\,{}_*|\, g\right] \\
&=
  \left[f \,|_*\, \overline{f} \meet \overline{g^\circ}\right]
  \left[\overline{f} \meet \overline{g^\circ} \,{}_*|\, g\right] \\
&=
  \left[f \,|_*\, \overline{f}\,\overline{g^\circ}\right]
  \left[\overline{f}\,\overline{g^\circ} \,{}_*|\, g\right]
= f \overline{f}\,\overline{g^\circ}\,\overline{f}\,\overline{g^\circ} g
= f \overline{f}\,\overline{g^\circ} g
= f g
\end{align*}
\end{proof}

\begin{proposition} \label{prop:lifunctor_preserve_tensor}
Locally inductive functors preserve tensors.
\end{proposition}
\begin{proof}
This follows immediately from the definition of a locally inductive
functor and the fact that any ordered functor preserves restrictions and
corestrictions \cite[Proposition 4.1.2(1)]{lawson98}.
\end{proof}

\begin{proposition}
\label{prop:groupoidfunctoriality}
For each functor $F:\mathbf{X}\rightarrow \mathbf{Y}$ between inverse
categories, there exists a locally inductive functor
$\mathcal{G}(F):\mathcal{G}(\mathbf{X})\rightarrow
 \mathcal{G}(\mathbf{Y}).$
\end{proposition}

\begin{proof}
We claim that $F:\mathbf{X}\rightarrow \mathbf{Y}$ induces a locally
inductive functor $\mathcal{G}(F)$ between the groupoids
$\mathcal{G}(\mathbf{X})$ and $\mathcal{G}(\mathbf{Y}).$ Since $F$ is a
functor of inverse categories, we have, for each $\overline{f}$ in
$\mathbf{X},$ that $F\overline{f} = \overline{F(f)}$ is a restriction
idempotent in $\mathbf{Y}.$ We can then define, for any object
$\overline{f}$ in $\mathcal{G}(\mathbf{X}),$
$\mathcal{G}(F)(\overline{f}) = F\overline{f}$ and this is a
well-defined object function.

Given an arrow $f:\overline{f}\rightarrow \overline{f^\circ}$ in
$\mathcal{G}(\mathbf{X}),$ we define
\[ \mathcal{G}(F)(f) :=
      \left[ F(f): F\left(\overline{f}\right) \rightarrow
        F\left(\overline{f^\circ}\right) \right] =
      \left[  F(f) : \overline{F(f)} \rightarrow
        \overline{F(f^\circ)} \right]. \]
We check that this is indeed an arrow in $\mathcal{G}(\mathbf{Y}).$
Clearly, $F(f)$ has the correct domain. We check, then, that it has the
correct codomain; that is, we verify that
$\overline{(F(f))^\circ} = \overline{F(f^\circ)}.$ By Lemma
~\ref{lem:restriction_functor_props}
~(\ref{lem:restriction_functor_props_isos}),
$(F(f))^\circ = F(f^\circ).$ It follows, then, that
$\overline{(F(f))^\circ} = \overline{F(f^\circ)}$ and thus $F$ is well
defined on arrows.

Since the objects of $\mathcal{G}(\mathbf{X})$ are specific arrows in
$\mathbf{X}$ and the composition in $\mathcal{G}(\mathbf{X})$ is, when
defined, given by composition in $\mathbf{X},$ the functoriality of
$\mathcal{G}(F)$ follows from the functoriality of $F.$

We check now that $F$ is an ordered functor. That is, we must check that
$F$ preserves partial orders. Suppose that $f\leq g$ are arrows in
$\mathcal{G}(\mathbf{X}).$ Then $g\overline{f} = f$ and thus
\[F(g)\overline{F(f)} = F(g)F(\overline{f}) = F(g\overline{f}) = F(f).\]
Therefore, $F(f)\leq F(g)$ in $\mathcal{G}(\mathbf{Y})$ and $F$ is an
ordered functor.

Finally, we verify that $F$ is a locally inductive functor. If
$\overline{a} \meet \overline{b}$ exists in $\mathcal{G}(\mathbf{X}),$
then $\overline{a}$ and $\overline{b}$ are endomorphisms on the same
object and are thus composable and in the same meet-semilattice. Then,
by the functoriality of $F,$
$F\left(\overline{a}\meet \overline{b}\right)
  = F\left(\overline{a}\,\overline{b}\right)
  = F(\overline{a})F\left(\overline{b}\right)
  = F(\overline{a})\meet F\left(\overline{b}\right).$
\end{proof}

\begin{corollary} \label{cor:groupoid_functor}
Construction~\ref{conts:orderedgroupoid} is the object function of a
fully faithful functor
$\mathcal{G}:\mathbf{iCat}\rightarrow \mathbf{tliGrpd}.$
\end{corollary}
\begin{proof}
By the proof of Proposition~\ref{prop:groupoidfunctoriality},
$\mathcal{G}$ is clearly a faithful functor.

Let $\mathbf{X}$ and $\mathbf{X'}$ be inverse categories and suppose
that $F: \mathcal{G}(\mathbf{X})\rightarrow \mathcal{G}(\mathbf{X'})$ is
a locally inductive functor. We seek, then, a functor
$F' : \mathbf{X}\rightarrow \mathbf{X'}$ with $\mathcal{G}(F') = F.$

For any two restriction idempotents $\overline{e}$ and $\overline{f}$ in
$E_A,$ we have $F(\overline{e} \meet \overline{f}) =
F\overline{e} \meet F\overline{f}$ since $F$ is locally inductive. This
implies that $F\overline{e}$ and $F\overline{f}$ are
$\mathbf{X'}$-endomorphisms on the same object and thus
$F(E_A) \subseteq E_B$ for some object $B\in \mathbf{X'}.$ So we can
define, for each object $A\in \mathbf{X},$ $F'(A)$ to be the object in
$\mathbf{X'}$ satisfying
$F(E_A) \subseteq E_{F'(A)}$ in $\mathcal{G}(\mathbf{X'}).$

Given any arrow $f: A\rightarrow B$ in $\mathbf{X},$ we must define an
arrow $F'(f) : F'(A) \rightarrow F'(B)$ in $\mathbf{X'}.$ We know that
$f$ corresponds to the arrow
$f: \overline{f} \rightarrow \overline{f^\circ}$ in
$\mathcal{G}(\mathbf{X}),$ whose image under $F$ is
$F(f) : F\overline{f} \rightarrow F\overline{f^\circ}$ in
$\mathcal{G}(\mathbf{X'}).$ Since $F\overline{f} \in F(E_A)$ and
$F\overline{f^\circ} \in F(E_B),$ this $F(f)$ corresponds to an arrow
$F'(f) : F'(A) \rightarrow F'(B)$ in $\mathbf{X'}.$

Clearly, identity arrows in $\mathbf{X},$ corresponding to identity
arrows in $\mathcal{G}(\mathbf{X})$ and mapped to identities in
$\mathcal{G}(\mathbf{X'})$ under $F,$ will be mapped to identities in
$\mathbf{X'}$ under $F'.$ We check that composition is preserved.
Suppose that $f$ and $g$ are arrows whose composite $gf$ exists in
$\mathbf{X}.$ Both $g$ and $f$ correspond, then, to arrows
$g:\overline{g}\rightarrow \overline{g^\circ}$ and
$f : \overline{f} \rightarrow \overline{f^\circ},$ respectively, in
$\mathcal{G}(\mathbf{X}).$ Notice that the composite $gf$ does not
necessarily exist (as an arrow) in $\mathcal{G}(\mathbf{X}),$ but that,
since $\overline{g},\overline{f^\circ}\in E_B,$ the tensor $g\otimes f$
does and that this tensor product uniquely corresponds to $gf$ by
Proposition~\ref{lem:tensor_is_composition}.
By Proposition~\ref{prop:lifunctor_preserve_tensor} (since $F$ preserves
meets), then, $F(g\otimes f) = F(g)\otimes F(f)$ and, again by
Lemma~\ref{lem:tensor_is_composition} and the definition of $F',$
corresponds to $F'(g) F'(f).$
\end{proof}

\begin{construction}
\label{const:inverserestriction}
Given a top-heavy locally inductive groupoid\\
$\left(\mathbf{G},\bullet,\leq, \{M_i\}_{i\in I}\right),$ define an
inverse category
$\left(\mathcal{I}(\mathbf{G}), \circ, \overline{(-)} \right)$ with the
following data:
\begin{itemize}
\item Objects: The objects are the meet-semilattices $M_i.$
\item Arrows: $\mathcal{I}(\mathbf{G})(M_1,M_2) =
  \{ f: A_1 \rightarrow A_2 \mbox{ in } \mathbf{G} \,|\,
     A_1\in M_1,\, A_2 \in M_2 \}.$
  Note that every object of $\mathbf{G}$ is in some $M_i,$ and the
  $M_i$ are disjoint, so that every arrow in $\mathbf{G}$ will be found
  in exactly one of these hom-sets.
\begin{itemize}
\item Composition: A composable pair of arrows $f: M_1 \rightarrow M_2$
and $g: M_2 \rightarrow M_3$ in $\mathcal{I}(\mathbf{G}),$ corresponds
to a pair of arrows $f: A_1 \rightarrow A_2$ and
$g: A'_2 \rightarrow A_3$ in $\mathbf{G}$ with
$A_1 \in M_1,$ $A_2, A'_2 \in M_2$ and $A_3 \in M_3.$ Since $M_2$ is a
meet-semilattice, the meet $A_2 \meet A'_2$ exists. We can therefore
define the composite of $f$ with $g$ as
$g\circ f = g\otimes f
          = [g\, |_*\, A_2\meet A'_2][A_2\meet A'_2\, {}_*|\, f].$
This composition is associative by
Proposition~\ref{prop:associative_tensor}.
\begin{figure}[ht]
  \centering
  \includegraphics[scale=1]{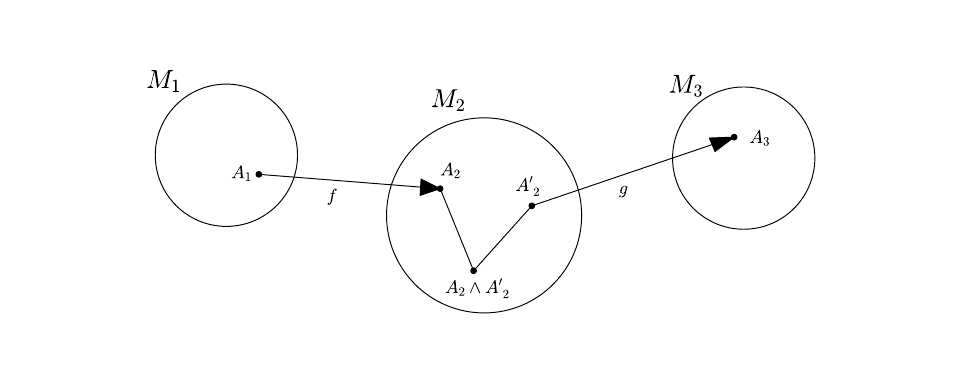}
\end{figure}
\item Identities: For each object $M_1,$ define
$1_{M_1}: M_1 \rightarrow M_1$ to be
$1_{\top_1} = \top_1 \rightarrow \top_1$ in $\mathbf{G}.$
Let $f: M_1 \rightarrow M_2$ be an arrow corresponding to
$f:A_1 \rightarrow A_2$ in $\mathbf{G}.$ Note that
$[1_{\top_1} \,|_*\, A_1 \meet \top_1] = 1_{A_1}$ by
Proposition~\ref{prop:identity_restrictions}. Then
\[
  f \circ 1_{\top_1} =
  \left[ f \,|_*\, A_1 \meet \top_1 \right] \bullet
    \left[ A_1 \meet \top_1 \,{}_*|\, 1_{\top_1} \right] =
  \left[ f \,|_*\, A_1 \right] \bullet 1_{A_1} = f. \]
Similarly, $1_{\top_2}\circ f = f.$
\begin{figure}[H]
  \centering
  \includegraphics[scale=1]{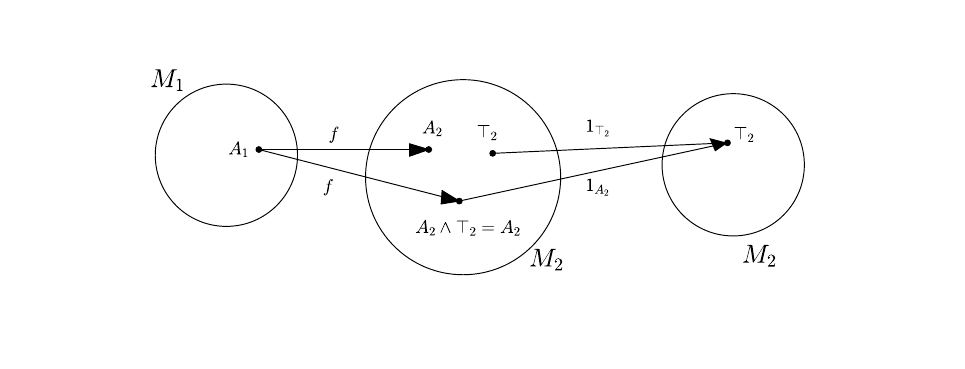}
\end{figure}
\item Restrictions: Given an arrow $f: M_1 \rightarrow M_2$
corresponding to an arrow $f: A_1\rightarrow A_2$ in $\mathbf{G},$
define
$\overline{f}: M_1 \rightarrow M_1$ by
$\overline{f} = 1_{A_1} : A_1 \rightarrow A_1.$
Conditions (R.1) -- (R.4) saying that $\mathcal{I}(\mathbf{G})$ is a
restriction category follow readily from the fact that all restriction
idempotents are identities on some object in $\mathbf{G}$ and that
restrictions in an ordered groupoid are unique.
\item Partial Isomorphisms: For each arrow $f: M_1 \rightarrow M_2,$
define $f^\circ: M_2 \rightarrow M_1$ as $f^{-1} : A_2 \rightarrow A_1.$
To check that this is a restricted inverse, we check the required
composites. First,
\[
f\circ f^\circ = f\otimes f^\circ =
[f \,|_*\, A_1\meet A_1] \bullet [A_1\meet A_1 \,{}_*|\, f^{-1}] =
f\bullet f^{-1} = 1_{A_2} = \overline{f^{-1}}. \]
Similarly, $f^\circ \circ f = \overline{f}.$
\end{itemize}
\end{itemize}
\end{construction}

\begin{proposition}
\label{prop:inverserestrictionfunctoriality}
For each locally inductive functor $F:\mathbf{G}\rightarrow \mathbf{H},$
there exists a functor
$\mathcal{I}(F):\mathcal{I}(\mathbf{G})\rightarrow
  \mathcal{I}(\mathbf{H}).$
\end{proposition}

\begin{proof}
We show that $F$ induces a functor
$\mathcal{I}(F): \mathcal{I}(\mathbf{G})\rightarrow
  \mathcal{I}(\mathbf{H}).$

Given any object in $\mathcal{I}(\mathbf{G}),$ a meet-semilattice
$M_1,$ define $\mathcal{I}(F)(M_1)$ to be \emph{the} meet-semilattice
$M'_1$ such that $F(M_1) \subseteq M'_1.$ Note that this assignment of
$M'_1$ to $M_1$ is unique since the $M'_i$ are a partition of
$\mathbf{H}_0.$

For any arrow $f:M_1\rightarrow M_2$ in $\mathcal{I}(\mathbf{G})$
corresponding to $f:A_1\rightarrow A_2$ in $\mathbf{G},$ we define
$\mathcal{I}(F)(f) = F(f) : F(A_1) \rightarrow F(A_2),$ an arrow
$F(f) : F(M_1) \rightarrow F(M_2)$ in $\mathcal{I}(\mathbf{G'}).$ That
this assignment is functorial follows from the functoriality of $F.$
\end{proof}

\begin{corollary} \label{cor:ircat_functor}
Construction~\ref{const:inverserestriction} is the object function of a
functor
\[\mathcal{I}: \mathbf{tliGrpd}\rightarrow \mathbf{iCat}.\]
\end{corollary}
\begin{proof}
Let
$\xymatrix@1{\mathbf{G} \ar[r]^{F}
  & \mathbf{G'} \ar[r]^{G} & \mathbf{G''}}$
be a composable pair of locally inductive functors. Then, on objects of
$\mathcal{I}(\mathbf{G})$ (meet-semilattices forming the partition of
$\mathbf{G_0}),$
\begin{align*}
\mathcal{I}(G)\mathcal{I}(F)(M)
&= \mathcal{I}(G)(M'), \mbox{ where } M'
   \mbox{ such that } FM \subseteq M' \\
&= M'', \mbox{ where } M''
   \mbox{ such that } M'' \supseteq G(M') = G(FM) = (GF)M \\
&= \mathcal{I}(GF)(M),\mbox{ by the uniqueness of } M'' \supseteq (GF)M.
\end{align*}
Equality of the functors $\mathcal{I}(GF)$ and
$\mathcal{I}(G)\mathcal{I}(F)$ follows immediately.  That $\mathcal{I}$
preserves identity functors follows from the observation that
$\mathcal{I}(1_\mathbf{G})(M) = M$ for all objects $M$ in
$\mathcal{I}(\mathbf{G}).$
\end{proof}

\begin{theorem} \label{thm:ircat_og_equivalence}
The functors $\mathcal{G}$ and $\mathcal{I}$ form an equivalence of
categories,
\[
\xymatrix{
\mathbf{iCat} \ar[r]<+0.3pc>^-{\mathcal{G}}
  & \mathbf{tliGrpd} \ar[l]<+0.3pc>^-{\mathcal{I}}
}
\]
\end{theorem}
\begin{proof}
By Corollary~\ref{cor:groupoid_functor}, the functor $\mathcal{G}$ is
fully faithful. We show now that $\mathcal{G}$ is essentially surjective
by demonstrating a natural isomorphism
$\mathcal{G}\mathcal{I} \cong 1_{\mathbf{tliGrpd}}.$

We start with a top-heavy locally inductive groupoid
$\left(\mathbf{G},\bullet,\leq, \{M_i\}_{i\in I}\right)$ and we consider
the composite $\mathcal{G}\mathcal{I}(\mathbf{G}).$ Recall that
$\mathcal{I}(\mathbf{G})$ has as objects the meet-semilattices $M_i$ and
arrows of the form $f:M_1 \rightarrow M_2,$ where
$f:A_1 \rightarrow A_2$ is an arrow in $\mathbf{G}$ with $A_1 \in M_1$
and $A_2 \in M_2.$ Further recall that every arrow in $\mathbf{G}$ is
found exactly once in $\mathcal{I}(\mathbf{G}).$ Note that for each
object $M_i,$
\[
  E_{M_i}
  = \{\overline{f}:M_i\rightarrow M_i | f: M_i \rightarrow M_i \}
  = \{ 1_{A_i} | A_i \in M_i \} \cong M_i.\]
Then the locally inductive groupoid $\mathcal{G}\mathcal{I}(\mathbf{G})$
contains the following data:
\begin{itemize}
\item Objects:
  $\displaystyle \coprod_{i\in I} E_{M_i} \cong  \coprod_{i\in I} M_i
    = \mathbf{G}_0.$
\item Arrows:
  For each $f: M_1 \rightarrow M_2$ in $\mathcal{I}(\mathbf{G})$
  corresponding to $f:A_1 \rightarrow A_2$ in $\mathbf{G},$ there is an
  arrow
  $f : \overline{f}\rightarrow \overline{f^\circ}
  = f : 1_{A_1} \rightarrow 1_{A_2} \cong f:A_1 \rightarrow A_2$ in
  $\mathcal{G}\mathcal{I}(\mathbf{G}).$ Since arrows of $\mathbf{G}$
  are appearing exactly once in $\mathcal{I}(\mathbf{G}),$ we have,
  then, that $(\mathcal{G}\mathcal{I}(\mathbf{G}))_1 \cong\mathbf{G}_1.$
\begin{itemize}
\item Composition:
  Given two composable arrows corresponding to $f:A_1 \rightarrow A_2$
  and $g : A_2 \rightarrow A_3$ in $\mathcal{G}\mathcal{I}(\mathbf{G}),$
  we have in $\mathcal{I}(\mathbf{G})$ that
  \[
  g \circ f = g\otimes f
  = [g\,|_*\, A_2 \meet A_2]\bullet [A_2 \meet A_2 \,{}_*|\, f]
  = g\bullet f.
  \]
  Their composite, then, is
  \[
  g\star f \mbox{ in } \mathcal{G}\mathcal{I}(\mathbf{G})
  = g \circ f  \mbox{ in } \mathcal{I}(\mathbf{G})
  = g\bullet f \mbox{ in } \mathbf{G}.
  \]
  That is, composition in $ \mathcal{G}\mathcal{I}(\mathbf{G}) $ is the
  same as that in $\mathbf{G}$ up to isomorphism.
\item Restrictions: Given an arrow
  $f: 1_{A_1} \rightarrow 1_{A_2} \cong
   f: A_1 \rightarrow A_2$ and $A'_1 \leq A_1,$ we have that
  \begin{align*}
  (f \,|_*\, A'_1) \mbox{ in } \mathcal{G}\mathcal{I}(\mathbf{G})
  &\cong f \circ 1_{A'_1} \mbox{ in } \mathcal{I}(\mathbf{G})
  = f \otimes 1_{A'_1} \mbox{ in } \mathbf{G} \\
  &= [f \,|_*\, A_1 \meet A'_1] \bullet
     [ A_1 \meet A'_1 \,|_*\, 1_{A'_1}]  \\
  &= [f \,|_*\, A'_1] \bullet 1_{A'_1} = [f \,|_*\, A'_1].
  \end{align*}
  That is, the restrictions of the two ordered groupoids $\mathbf{G}$
  and $\mathcal{G}\mathcal{I}\mathbf{G}$ are the same up to isomorphism.
\end{itemize}
\end{itemize}
This description of $\mathcal{G}\mathcal{I}(\mathbf{G})$ is written so
that the isomorphism
$\mathbf{G}\cong\mathcal{G}\mathcal{I}(\mathbf{G})$ follows immediately.
\end{proof}

\begin{note}
In an inverse semigroup $(S,\bullet),$ every idempotent is of the form
$s^\bullet \bullet s$ for some $s\in S.$ In addition, all idempotents
commute. We can then consider the groupoid associated to an inverse
semigroup as the Karoubi envelope of the single-object inverse category
(with unit) associated to $S.$ In a general inverse category, this
fact ensures that every restriction idempotent will appear as an object
in the associated top-heavy locally inductive groupoid, and that every
object in this groupoid is a restriction idempotent.
\end{note}

The definition of the functor $\mathcal{G}$ relies on the top-heavy
property of a locally inductive groupoid $\mathcal{G}$ only when
defining identities on the meet-semilattices partitioning
$\mathbf{G}_0.$ Similarly, the identities of an inverse category
$\mathbf{X}$ are essential only as top elements of the meet-semilattices
$E_A.$ In other words, removing identities from an inverse category is
equivalent to removing top elements from the meet-semilattices
partitioning a locally inductive groupoid. As a result, the equivalence
established in Theorem~\ref{thm:ircat_og_equivalence} generalizes
immediately.

\begin{corollary} \label{cor:irscat_og_equiv}
The functors $\mathcal{G}$ and $\mathcal{I}$ form an equivalence
\[
\xymatrix{
\mathbf{isCat} \ar[r]<+0.3pc>^{\mathcal{G}}
& \mathbf{liGrpd} \ar[l]<+0.3pc>^{\mathcal{I}},
}
\]
where $\mathbf{isCat}$ is the category of inverse semicategories.
\end{corollary}

Since single-object inverse categories are precisely inverse semigroups
with identity, it is clear that single-object inverse semicategories are
precisely inverse semigroups. With inverse semicategories as
multi-object inverse semigroups, we see that Theorem~\ref{thm:ESN} --
the equivalence between inductive groupoids and inverse semigroups --
then follows immediately from Corollary~\ref{cor:irscat_og_equiv}.

We will end this section with a short discussion on a generalization of
Theorem~\ref{thm:ircat_og_equivalence}.

Recall that \textit{prehomomorphisms} of inverse semigroups are
functions between inverse semigroups satisfying
$\phi(ab) \leq \phi(a)\phi(b).$ Theorem~\ref{thm:ESN} can then be
generalized to
\begin{theorem}[\cite{lawson98}, Theorem 8]
The category of inverse semigroups and prehomomorphisms is equivalent to
the category of inductive groupoids and ordered functors.
\end{theorem}

Since the arrows of an inverse category are playing the part of
``elements'' in each of the ``local inverse semigroups'', a clear
candidate for an inverse categorical analogue arises.
\begin{definition}
An \emph{oplax functor} $F : \mathbf{X} \rightarrow \mathbf{X'}$  of
inverse categories consists of the following data:
\begin{itemize}
\item for each object $A\in\mathbf{X},$ an object $F(A)\in\mathbf{X'};$
\item for each arrow $f : A \rightarrow B,$ an arrow
      $F(f) : F(A) \rightarrow F(B)$ such that
      \begin{itemize}
      \item for each composable pair $f : A \rightarrow B$ and
            $g : B\rightarrow C$ in $\mathbf{X},$ $F(gf) \leq F(g)F(f),$
            and
      \item for each object $A\in \mathbf{X},$ $F(1_A) \leq 1_{F(A)}.$
      \end{itemize}
\end{itemize}
\end{definition}

Clearly, since composition in $\mathcal{G}(\mathbf{X})$ is defined by
composition in $\mathbf{X},$ any oplax functor
$F: \mathbf{X} \rightarrow \mathbf{X'}$ between inverse categories
induces an ordered functor
$\mathcal{G}(F) : \mathcal{G}(\mathbf{X}) \rightarrow
 \mathcal{G}(\mathbf{X'}).$

Suppose now that $F: \mathbf{G} \rightarrow \mathbf{G'}$ is an ordered
functor between top-heavy locally inductive groupoids. Recall that
composition in $\mathcal{I}(\mathbf{G})$ is defined by the tensor
product in $\mathbf{G}.$ Then
\begin{align*}
F(g\otimes f)
&= F(g\, |_* \, \dom(g) \meet
   \cod(f)) F( \dom(g) \meet \cod(f) \, {}_*| \, f) \\
&= (Fg\, |_* \, F(\dom(g) \meet \cod(f))) ( F(\dom(g) \meet
   \cod(f)) \, {}_*| \, Ff) \\
&\leq (Fg\, |_* \, F(\dom(g) \meet F\cod(f)) ( F\dom(g) \meet
   F\cod(f) \, {}_*| \, Ff) \\
&= Fg \otimes Ff
\end{align*}
and thus $F$ induces an oplax functor
$\mathcal{I}(F) : \mathcal{I}(\mathbf{G}) \rightarrow
 \mathcal{I}(\mathbf{G'}).$
 Specifically, since the identities in $\mathcal{I}(\mathbf{G})$ are the
 top elements of $\mathbf{G},$ $\mathcal{I}(F)$ is strict on identities.

These arguments can then be easily extended to prove the following.

\begin{theorem} \label{thm:oplax_is}
The category of top-heavy locally inductive groupoids and ordered
functors is equivalent to the category of inverse categories and
oplax functors.
\end{theorem}

\begin{note}
Since the 2-category structure of an inverse category is posetal,
pseudofunctors (oplax functors whose 2-cells are isomorphisms) are
exactly ordinary functors between inverse categories. The category of
inverse categories and pseudofunctors is therefore equal to the category
of inverse categories and ordinary functors, and equivalent to the
category of top-heavy locally inductive groupoids and locally inductive
functors.
\end{note}

\bibliographystyle{plain}
\bibliography{tac.bbl}

\end{document}